\documentclass[11pt]{amsart}
\usepackage{amssymb}

\usepackage{palatino}
\input amssym.def

\usepackage{amsmath, amsfonts}
\usepackage{amssymb}
\usepackage{amscd}
\usepackage[mathscr]{eucal}
\usepackage{graphicx}
\usepackage{palatino}
\newfont{\cyrr}{wncyr10}

\newcommand{\N}{{\mathbb N}}
\newcommand{\Z}{{\mathbb Z}}
\newcommand{\Q}{{\mathbb Q}}
\newcommand{\R}{{\mathbb R}}
\newcommand{\C}{{\mathbb C}}

\renewcommand{\P}{{\mathbb P}}

\renewcommand{\mod}{{\, \rm mod \, }}

\newtheorem{thm}{Theorem}

\newtheorem{lem}[thm]{Lemma}
\newtheorem{cor}[thm]{Corollary}
\newtheorem{prop}[thm]{Proposition}

\newtheorem{rmk}[thm]{Remark} 

\newtheorem{conj}[thm]{Conjecture}
\newtheorem{exmpl}[thm]{Example}

\def\ndiv{{\kern3pt | \kern-6pt - \kern1pt}}

\begin{document}

\title[Mahler]{ On two problems of Hardy and  Mahler}

\author{Patrice Philippon and Purusottam Rath}



\begin{abstract} 
It is a classical  result of Mahler that for any  rational number $\alpha > 1$ which is not an integer and any real $ 0 < c <1$, the set of positive integers  $n$ such that $\Vert\alpha^n\Vert < c^{n}$ is necessarily finite. Here for any real $x$, $\Vert x\Vert$ denotes
the distance from its nearest integer.  The problem of classifying
all real algebraic numbers greater than one exhibiting the above phenomenon was 
suggested by Mahler. This was solved by a beautiful work of
Corvaja and Zannier.  On the other hand, for non-zero real numbers  $\lambda$
and  $\alpha$  with  $\alpha >1$, Hardy  about a century ago asked \\
$\phantom{mamma}$ ``In what circumstances can it be true that $\Vert \lambda \alpha^n\Vert  \to 0$ as $n\to \infty$? " \\
This question is still open in general.  In this note, we study its analogue in
the context of the problem of Mahler. We first compare and contrast with what
is known vis-a-vis the original question of Hardy. We then suggest a number of questions that arise as natural consequences of our investigation. Of these questions, we answer one and offer some insight into others.
\end{abstract}

\maketitle   

\section{introduction}
Throughout the paper, $\N$ and $\N^{\times}$  denote the set of non-negative integers
and positive integers respectively. Further for any real $x$,  $\Vert x\Vert$ denotes
its distance from its nearest integer. 
In other words,
$$
\Vert x\Vert  := {\rm min} \{ |x - m| : m \in \Z\}.
$$
The growth of the sequence $\Vert (3/2)^n\Vert $
is intricately linked to the famous Waring's problem.
For  a positive integer $k$, let 
$g(k)$ denote the minimum number $s$ such that
every positive integer is expressible as a sum of $s$
$k$-th powers of elements in $\N$.  It is known that  for $ k \geq 6$,
$$
g(k) = 2^k + [(3/2)^k] -2
$$
if $\Vert (3/2)^k\Vert   \geq (3/4)^k$.

This was the motivation for
Mahler \cite{mahler} in 1957 to prove that  for any  $\alpha \in \Q\setminus \Z$, $\alpha>1$,
and any real $ 0 < c <1$, the set of $n \in \N$ such that
$\Vert \alpha^n\Vert  < c^{n}$ is finite.  Mahler uses  Ridout's theorem
which is a $p$-adic extension of the famous theorem of Roth
that algebraic irrationals have irrationality measure two.

Mahler ends his  paper  suggesting  that
it would be of some interest to know which algebraic numbers have the same property 
as the non integral rationals, that is to  characterise  real algebraic numbers $\alpha > 1$ such
that  for any real $ 0 < c <1$, the set of $n \in \N$ such that
$\Vert \alpha^n\Vert  < c^{n}$ is finite.

This was answered completely by Corvaja and Zannier \cite{cz} who proved the following:
\begin{thm}\label{czthm} {\bf  ( Corvaja and Zannier ) }
Let $\alpha > 1$ be a real algebraic number. Then for  some  $ 0 < c < 1$,
$\Vert \alpha^n\Vert  < c^{n}$ for infinitely many $n \in \N$
 if and only if  there exists 
a positive integer $s$ such that $\alpha^s$ is a PV number.
In particular, $\alpha$ is an algebraic integer.
\end{thm}

We recall that a PV number (for Pisot-Vijayaraghavan number) is a real algebraic integer $\alpha$
such that:
\begin{itemize}
 \item $\alpha >1$;
 \item all its other conjugates (over $\Q$) have absolute values
strictly less than $1$.
\end{itemize}

\smallskip
\noindent

On the other hand,  for non-zero real numbers  $\lambda$
and  $\alpha$  with  $\alpha >1$, Hardy \cite{hardy}  about a century ago asked \\
$\phantom{mamma}$ ``In what circumstances can it be true that $\Vert \lambda \alpha^n\Vert  \to 0$ as $n\to \infty$? " \\
This question is still open in general.  However, when $\alpha$ is further assumed to 
be algebraic, then the question is easier and can be settled as 
was shown by Hardy himself.  More precisely, one has ( Theorem A in \cite{hardy}):

\begin{thm}\label{hardy} {\bf ( Hardy )}
Let $\alpha >1$ be an algebraic  real number and   $\lambda$
be a non-zero real number. Suppose that
 $\Vert \lambda \alpha^n\Vert  \to 0$ as $n\to \infty$. Then
 $\alpha$ is a PV number, hence an algebraic integer.
Further in this case, $\lambda$ necessarily lies in $\Q(\alpha)$
and hence is algebraic.
\end{thm}

In this note, we consider the analogue of this question in
the context of the problem of Mahler. More precisely,  consider the sets
$$ H := \{ (\lambda, \alpha) \in \R^{\times} \times  (1, \infty) :~~~ \Vert \lambda \alpha^n\Vert  \to 0 ~~~\mbox{as}~~~ n \to \infty\}
$$
and
$$
M := \{ (\lambda, \alpha) \in \R^{\times} \times (1, \infty) : \exists~
0 < c < 1 ~~ \mbox{such ~~~that}~~
\Vert \lambda \alpha^n\Vert  < c^{n} ~~~ \mbox{for~~ infinitely~~~ many} ~~~~~n \in \N\}.
$$
The goal of our work is to compare and contrast these two sets  lying in $\R^2$.
Sometimes we refer to these sets as the Hardy set and the Mahler set respectively.

To start with, we note that it is known that the set $H$ is countable. See for instance,
the pretty book of Salem \cite[Chap.1, \S4]{salem}. 

On the other hand, the set $M$ is uncountable. In fact, for any sequence $(\epsilon_n)_{n\in\N}$ in $\{0, 1\}$ not identically $0$, set $\lambda = \sum_{n=0}^{\infty} \frac{\epsilon_n}{10^{n!}}$, then $(\lambda, 10)$ lies in $M$.

We also note that Corvaja and Zannier in their
paper construct a transcendental real number $\alpha > 1$ such that $(1, \alpha)$ lies in $M$. Further, it follows from a work of  Bugeaud and Dubickas \cite{bugeaud} that there are uncountably many $\alpha > 1$ such that $(1, \alpha)$ lies in $M$.
 
In this context, one has the following folklore conjecture, which proposes an answer to Hardy's question.
\begin{conj}\label{conjfolk}
If $(\lambda, \alpha) \in H$, then $\alpha$ is a PV number and  $\lambda$ lies in $\Q(\alpha)$.
\end{conj}

Clearly, this no longer holds for the set $M$ as  there 
are transcendental numbers $\alpha > 1$ such that $(1, \alpha)$ lie in $M$.

More interestingly, even   if $(\lambda, \alpha) \in M$ with $\alpha$
algebraic, it does not imply that $\lambda$ is algebraic. 
For instance, as remarked before, $(\lambda,10)\in M$ with $\lambda$ the Liouville number $\sum_{n=1}^{\infty} \frac{1}{10^{n!}}$. This is in contrast to Theorem \ref{hardy}.

\begin{rmk}  More generally, when $\alpha > 1 $ is an integer and $\lambda$ is a non zero real number, $(\lambda,\alpha)\in M$ if and only if for some $\varepsilon>0$, the sequence of digits of the expansion of $\lambda$ in base $\alpha$ contains infinitely many blocks of the form $x_n \cdots x_{n+[\varepsilon n]}$ consisting only of zeros or only of $(\alpha-1)$'s.
\medskip
\end{rmk}

These observations give rise to a number of questions.

\bigskip
\noindent
{\bf Questions:}

{\it 
1. What can be said about $(\lambda, \alpha) \in M$ where both
$\lambda$ and $\alpha$ are assumed to be algebraic? Can one
have a theorem similar to Theorem \ref{hardy} in this case?

2. Is it possible to derive some diophantine characterisation of transcendental numbers $\lambda$ such that $(\lambda, \alpha) \in M$ for some algebraic $\alpha$?
 
3.  Is it true that $(1, \alpha) \in H$ implies that
$(1, \alpha) \in M$? This is a consequence of the following.

4. Is it true that $H \subset M$ ?  

5. What is the Hausdorff dimension of the set  $M$?

}
\medskip

As introduced in \cite{cz}, a pseudo-PV number is 
an algebraic number $\alpha \in \C$ such that:
\begin{itemize}
\item $| \alpha | >1$,
 \item all its other conjugates have  absolute values strictly less than $1$;
 \item  $Tr_{\Q(\alpha)/\Q}(\alpha) \in \Z$. 
\end{itemize}

Note that a pseudo-PV number is necessarily real; a positive
pseudo-PV number which is an algebraic integer is simply 
a usual PV number.

Here is our first theorem.

\begin{thm}\label{2}
Let $\lambda$ and $\alpha$ be algebraic numbers such that $(\lambda, \alpha) \in M$. Then there exists a positive integer $s$ such that $\alpha^s$ is a PV number
and hence an algebraic integer.

Furthermore for any $h\in\N^\times$, there are infinitely many $n\in\N$ such that
$\lambda\alpha^n$ lies  in $\Q(\alpha^h)$, is a pseudo-PV number whose trace is non-zero and $\Vert\lambda\alpha^n\Vert<c^{n}$ for some $0<c<1$. 
\end{thm}

Proof of this theorem builds upon the techniques developed by Corvaja and Zannier.
We also need to prove some further diophantine results on behaviour of powers of Salem numbers
(Proposition \ref{propspv}) proof of which requires the $p$-adic subspace theorem. 

\medskip
\noindent
Recall that a Salem number is a real algebraic integer
$\alpha $ such that
\begin{itemize}
\item $\alpha >1$,
 \item all its other conjugates have  absolute values  at most $1$;
 \item at least
one conjugate has absolute value equal to $1.$
\end{itemize}

\medskip
\noindent

The distribution of exponential sequences $(\Vert \alpha^n\Vert)_{n\in\N}$  
is rather mysterious and the  few cases where one has some information 
is when
$\alpha$ is an algebraic integer and  $Tr_{k/\Q}(\alpha^n)$, with $k=\Q(\alpha)$,
is close to the nearest integer of $\alpha^n$.  This motivates 
the study of  algebraic numbers 
$\alpha$  such  that  $Tr_{k/\Q}(\alpha^n) \in \Z$
for $n$ lying in suitable subsets $I$ of $\N$. 

It is not difficult to see that  if $Tr_{k/\Q}(\alpha^n)$ is an integer for all $n \in \N$, $\alpha$ is 
an algebraic integer. For, the complementary  module $L'$ of the lattice
$L$ generated by $1, \alpha, \cdots, \alpha^d$ is a finitely generated
$\Z$-module and the hypothesis above implies that  $L'$ contains the ring
$\Z[\alpha]$. Hence $\alpha$ is necessarily an algebraic integer.

 On the other hand, one has the following nice result of Bart de Smit \cite{de}.
\begin{thm} {\bf (Bart de Smit )}
Let $\alpha$ be an algebraic number of degree $d$ such that $Tr_{k/\Q}(\alpha^i)$ is an
integer for all  natural numbers $i$  with $1 \leq i  \leq d + d \log_{2}d$.
Then $\alpha$ is an algebraic integer.
\end{thm}

 The example $\alpha = \frac{1}{\sqrt{2}}$ shows that the above bound is optimal.

\smallskip
\noindent
On the other hand, a minor modification of the works of Corvaja and Zannier  yields the following:
 
\begin{thm} \label{int1}
Let  $\alpha $ be an  algebraic number
such that $Tr_{k/\Q}(  \alpha^n) $
is a non-zero integer for infinitely many $n \in \N$.
 Then $\alpha$ is necessarily
an algebraic integer.
\end{thm}

This is an immediate consequence of  Lemma~\ref{3} in Section 2.
 \bigskip
 \noindent

Note that in the above theorem, the hypothesis that  $Tr_{k/\Q}(  \alpha^n) $ is non zero 
is necessary.  For   $\alpha = \frac{1}{\sqrt{2}}$ satisfies 
$Tr(\alpha^n) = 0$ for all odd $n$.
However, the example of  $\alpha = \frac{1}{\sqrt{2}}$
is not generic. More precisely, if  $\alpha$
is a real algebraic number such that
$Tr_{k/\Q}( \alpha^n) = 0 $ for infinitely many $n \in \N$,
then $\alpha$ is not necessarily the root of
a rational number. Here is an example.

\begin{exmpl}
The roots of the polynomial $X^4-6X^2+4$ are the real algebraic numbers $\pm\frac{1\pm\sqrt{5}}{\sqrt{2}}$. In particular, the splitting field of this polynomial is real and has only two roots of unity $\pm1$. Set $\alpha=\frac{1+\sqrt{5}}{\sqrt{2}}$ the largest root, we check $\alpha\notin\Q(\alpha^2)$ and $Tr_{\Q(\alpha)/\Q}(\alpha^{2m+1})=0$ for $m\in\N$. However, $\alpha$ is not a root of a rational number.
\end{exmpl}

In this context, we have the following theorem.

\begin{thm} \label{int3real}
Let $\alpha$ be a real, positive algebraic number, $h$ be the order of the torsion group of the splitting field of the minimal polynomial of $\alpha$. Then, $Tr_{\Q(\alpha)/\Q}(\alpha^n) = 0$ for infinitely many $n \in \N$ if and only if $\alpha$ does not belong to the field $\Q(\alpha^h)$. 
\end{thm}

We also derive the following theorem which works for more general algebraic numbers.

\begin{thm}\label{int3}
Let $\alpha$ be a nonzero algebraic number, $h$ be the order of the torsion group of the splitting field of the minimal polynomial of $\alpha$ and $\zeta$ a primitive $h$-th root of unity. Then, $Tr_{\Q(\alpha,\zeta)/\Q(\zeta)}( \alpha^n) = 0$ for infinitely many $n \in \N$ if and only if $\alpha$ does not belong to the field $\Q(\alpha^h,\zeta)$. 
\end{thm}

Proof of these results are given in the penultimate section ( Section 4 )
of the paper.

\medskip
\noindent
The final theme of the paper which constitutes the last section of our work is devoted to a more careful study of the algebraic elements in the Hardy and Mahler sets. For instance, we give the following more refined description of the  pairs of algebraic numbers in the Hardy set $H$.
\smallskip

\begin{thm} \label{hardy+}
The elements of the set $H\cap\overline{\Q}^2$ precisely consists of all the pairs $(\lambda,\alpha)$ where $\alpha$ is a PV number and $\lambda\in\frac{1}{P'_\alpha(\alpha)}\Z\left[\alpha,\frac{1}{\alpha}\right]$ with $P_\alpha\in\Z[X]$ the minimal polynomial of $\alpha$ over $\Z$.
\end{thm}
\smallskip

As for algebraic elements in the Mahler set $M$, we can show the following.

\smallskip
\begin{thm} \label{mahler+}
We have $(\lambda,\alpha)\in M\cap\overline{\Q}^2$ if and only if there exists integers $s$ and $0\le t<s$ such that $\alpha^s$ is a PV number and $\lambda$ belongs to $\frac{1}{\alpha^{t}P'_{\alpha^{s}}(\alpha^{s})}\Z\left[\alpha^{s},\frac{1}{\alpha^{s}}\right]$ with $P_{\alpha^{s}}\in\Z[X]$ the minimal polynomial of $\alpha^{s}$ over $\Z$.
\end{thm}

\begin{rmk}
The above theorems allow us to derive a result (see  Corollary \ref{corhardy+mahler+}) which is in the direction of the fourth question. Namely, we show that $H\cap\overline{\Q}^2\subset M$, which reduces the fourth question to Conjecture~\ref{conjfolk}.
\end{rmk}

\section{ Intermediate results }

We fix the notion of height of a non-zero algebraic number which we shall be working with. For any number field $K$, let $M_K$ be the set of all inequivalent places of $K$. The corresponding absolute values $|\cdot|_v$ are normalized so that $|\cdot|_v^{\frac{[K:\Q]}{[K_v:\Q_v]}}$ extends the usual archimedean or $p$-adic absolute value of $\Q$. Thus, the product formula for $\alpha\in K^\times$ reads $\prod_{v\in M_K}|\alpha|_v=1$ and the height $H(\alpha)$, defined as
$$
H(\alpha) := \prod_{v \in M_K} {\rm max}~\{1, |\alpha|_v \},
$$
is unambiguous and does not depend on the choice of $K$ containing $\alpha$.

\noindent
Similarly, for an integer  $n>1$ and  a non zero vector ${\overline{\alpha}} = (\alpha_1, \cdots, \alpha_n) \in K^n$, we set
$$
H(\overline{\alpha}) : = \prod_{v \in M_K} {\rm max}~\{|\alpha_1|_v, \cdots, |\alpha_n|_v \}.
$$

We will need the following lemma, which extends and improves \cite[Lemma 4]{cz}:
\begin{lem}\label{3}
Let $\lambda$ and $\alpha$ be algebraic numbers. Let $k := \Q(\lambda ,\alpha)$ and suppose that the trace $Tr_{k/\Q}( \lambda \alpha^n) $
is a non-zero integer for infinitely many $n$. Then $\alpha$ is necessarily
an algebraic integer.
\end{lem}
\begin{proof}
The proof follows from that  of \cite[Lemma 4]{cz} with some minor modifications.
We just need to work with the field $k= \Q(\lambda, \alpha)$. Finally, an extra argument ensures that $\alpha$ cannot be the root of a (non-integral) 
rational number and hence must be an algebraic integer.

Let $K$ be the Galois closure of the extension $k/\Q$ and $h$ the order of the torsion group of $K^\times$. Let $I\subset\N$ be the set of exponents such that $Tr_{k/\Q}( \lambda \alpha^n)$ is a non-zero integer. Since $I$ is infinite, there exists an integer $c\in\{0,\dots,h-1\}$ and an infinite set $J\subset\N$ such that $c+hm\in I$ for all $m\in J$. Let $d=[\Q(\alpha^h):\Q]$. 

We first deal with the case $d=1$. Then $\alpha^h=\frac{a}{b}$ where $a$ and $b$ are co-prime integers. We can write
$$
Tr_{k/\Q}( \lambda\alpha^{c+hm}) = Tr_{k/\Q}(\lambda\alpha^c)\left(\frac{a}{b}\right)^{m}.
$$
For $m\in J$ we have $c+hm\in I$, thus $Tr_{k/\Q}( \lambda\alpha^{c+hm})$ is a non-zero integer by hypothesis and $Tr_{k/\Q}( \lambda\alpha^c)$ is a fixed rational number. The above equality implies, as $m$ tends to infinity in $J$, $b=1$ and $\alpha$ is an algebraic integer (in fact a root of a rational integer).

In the case $d>1$, let $\sigma_1,\dots,\sigma_d$ be a complete set of representatives of ${\rm Gal}(K/\Q)$ modulo the subgroup fixing $\Q(\alpha^h)$. For $i=1,\dots,d$, let $T_i$ be a complete set of representatives of ${\rm Gal}(K/\Q)$ modulo the subgroup fixing $k$ that coincides with $\sigma_i$ modulo the subgroup fixing $\Q(\alpha^h)$. In particular, $T_1\cup\dots\cup T_d$ is a complete set of representatives of ${\rm Gal}(K/\Q)$ modulo the subgroup fixing $k$ and we can write
$$
Tr_{k/\Q}( \lambda\alpha^{c+hm}) = \sum_{i=1}^d\left(\sum_{\tau\in T_i}\tau(\lambda\alpha^c)\right)\sigma_i(\alpha^{hm}) = A_1\sigma_1(\alpha^{hm})+\dots+A_d\sigma_d(\alpha^{hm}),
$$
where we have set $A_i=\sum_{\tau\in T_i}\tau(\lambda\alpha^c)$. We now proceed by contradiction, assuming $|\alpha|_v>1$ for some finite place $v$ of $k$. For $m\in J$ we have $c+hm\in I$, thus $Tr_{k/\Q}( \lambda\alpha^{c+hm})$ is a non-zero integer by hypothesis, and the coefficients $A_i$ cannot be all zero. Furthermore, for $\varepsilon<\frac{\log|\alpha|_v}{\log H(\alpha)}$ and $m\in J$ sufficiently large we have
$$
|A_1\sigma_1(\alpha^{hm})+\dots+A_d\sigma_d(\alpha^{hm})|_v = |Tr_{k/\Q}( \lambda\alpha^{c+hm})|_v \le 1 < |\alpha^{hm}|_v H(\alpha^{hm})^{-\varepsilon}.
$$
Applying Lemma 1 of~\cite{cz} with $\Xi=\{\alpha^{hm};m\in J\}$ and $S$ a suitable Galois set of places of $K$ such that $\alpha$ is an $S$-unit, we obtain a non-trivial equation satisfied by infinitely many $m\in J$
$$
a_1\sigma_1(\alpha^h)^m+\dots+a_d\sigma_d(\alpha^h)^m = 0,\quad a_i\in K.
$$
We may then apply Skolem-Mahler-Lech's theorem \cite[Corollary 7.2, page 193]{cetraro} which entails that
$$\frac{\sigma_i(\alpha^h)}{\sigma_j(\alpha^h)} = \left(\frac{\sigma_i(\alpha)}{\sigma_j(\alpha)}\right)^h$$
is a root of unity for two distinct indices $i\not=j$. But this ratio must then be $1$ because $\frac{\sigma_i(\alpha)}{\sigma_j(\alpha)}$ is a root of unity in $K$ and $h$ is the exponent of the torsion group of $K^\times$. This implies that $\sigma_i$ and $\sigma_j$ coincide on $\Q(\alpha^h)$, contradicting their definition. This contradiction shows that $|\alpha|_v\le1$ for all finite place $v$ of $k$, hence $\alpha$ is an algebraic integer.
\end{proof}

We note that in the above theorem, a priori  there need not be  
any relation between the algebraic
numbers $\lambda$ and $\alpha$. In particular,
$\lambda$ need not be in $\Q(\alpha)$.

Finally, we note the following theorem 
which we shall need. This is a special case of a deep 
theorem  of Corvaja and Zannier ( \cite{cz}, see Main Theorem ).

\begin{thm}\label{cz2}
Let $\Gamma$ be a finitely generated subgroup of ~ $\overline{\Q}^{\times}$, $\lambda \in \overline{\Q}^{\times}$ and $\epsilon >0$ be real. Suppose that the following diophantine inequality
$$
0 < \Vert  \lambda  u\Vert  < H(u)^{- \epsilon}
$$
has infinitely many solutions $ u \in \Gamma$ with $|\lambda u| > 1$. Then all but finitely
many such $\lambda u$ are pseudo-PV numbers.
\end{thm}

\section{ Proof of  Theorem \ref{2}}

Proof of Theorem \ref{2} rests on the following intermediate results.

\begin{prop}\label{propspv}
Let $\alpha_1,\dots,\alpha_r$ be non-zero algebraic numbers of modulus $1$ and $\lambda_1,\dots,\lambda_r$ be non-zero algebraic numbers. For $\varepsilon>0$, assume that there exists an infinite set $I\subset\N^\times$ such that for all $n \in I$,
 $ \lambda_1\alpha_1^n+\dots+\lambda_r\alpha_r^n$ is a real number and 
$$\Vert \lambda_1\alpha_1^n+\dots+\lambda_r\alpha_r^n\Vert< {\rm e}^{-\varepsilon n}.$$
Then there exists algebraic numbers $A_0,A_1,\dots,A_r$ not all zero such that for infinitely many $n\in I$, one has
$$A_0+A_1\alpha_1^n+\dots+A_r\alpha_r^n = 0.$$
\end{prop}
\begin{proof}
We begin by noting that if one of the  $\alpha_1,\dots,\alpha_r$ is a root of unity, the conclusion holds.
So we may assume that none of the  $\alpha_1,\dots,\alpha_r$ is a root of unity.

Let $N_n$ denote the integer closest to $\lambda_1\alpha_1^n+\dots+\lambda_r\alpha_r^n$ and observe that $|N_n|\le|\lambda_1|+\dots+|\lambda_r|$ takes only finitely many values. Thus, there exists an infinite subset $J\subset I$ and an integer $\lambda_0$ such that $|\lambda_0+\lambda_1\alpha_1^n+\dots+\lambda_r\alpha_r^n| = \Vert\lambda_1\alpha_1^n+\dots+\lambda_r\alpha_r^n\Vert < {\rm e}^{-\varepsilon n}$ for $n\in J$. Set $\alpha_0=1$.

Let $K\subset\C$ be a Galois number field containing $\lambda_1,\dots,\lambda_r,\alpha_1,\dots,\alpha_r$ and $S$ a finite set of absolute values of $K$, containing all the archimedean ones and such that $\alpha_1,\dots,\alpha_r$ are $S$-units. Let $v_0$ be the (archimedean) absolute value given by the given inclusion of $K$ in $\C$. If $\lambda_0\not=0$ we set $i_0=0$ and otherwise $i_0=1$. We set
$$L_{v_0,i_0}(\overline{x}) = \lambda_0x_0+\dots+\lambda_rx_r$$
with $\overline{x}=(x_0,\dots,x_r)$ and for $(v,i)\in S\times\{0,\dots,r\}$, $(v,i)\not=(v_0,i_0)$, $L_{v,i}(\overline{x})=x_i$. Observe that for each $v$ the forms $L_{v,0},\dots,L_{v,r}$ are linearly independent. Then, with $\overline{\alpha}^n=(\alpha_0^n,\dots,\alpha_r^n)$ and $\Vert\overline{\alpha}\Vert_v=\max(|\alpha_0|_v,\dots,|\alpha_r|_v)$, we have
\begin{align*}
\prod_{v\in S}\prod_{i=0}^r\frac{|L_{v,i}(\overline{\alpha}^n)|_v}{\Vert\overline{\alpha}^n\Vert_v}
&= \frac{|L_{v_0,i_0}(\overline{\alpha}^n)|_{v_0}}{|\alpha_{i_0}|_{v_0}^n}\prod_{i=0}^r\prod_{v\in S}\frac{|\alpha_i|_v^n}{\Vert\overline{\alpha}\Vert_v^n}\\
&= \frac{|L_{v_0,i_0}(\overline{\alpha}^n)|}{|\alpha_{i_0}|^n}H(\overline{\alpha})^{-n(r+1)},
\end{align*}
since $\prod_{v\in S}|\alpha_i|_v=1$ by the product formula, and $\prod_{v\in S}\Vert\overline{\alpha}\Vert_v = H(\overline{\alpha})$ because the components of $\overline{\alpha}$ are $S$-units. 

For $n\in J$ we further have $|L_{v_0,i_0}(\overline{\alpha}^n)|=\Vert\lambda_1\alpha_1^n+\dots+\lambda_r\alpha_r^n\Vert<{\rm e}^{-\varepsilon n}$ and we observe that $|\alpha_{i_0}|=1$ in any case, thus
$$\prod_{v\in S}\prod_{i=0}^r\frac{|L_{v,i}(\overline{\alpha}^n)|_v}{\Vert\overline{\alpha}^n\Vert_v} < H(\overline{\alpha}^n)^{-r-1-\varepsilon /\log H(\overline{\alpha})}.$$
 
The $p$-adic subspace theorem~\cite[Chap. V, Thm. 1D', page 178]{schmidt} then ensures that for all $n\in J$ the point $(\alpha_0^n:\dots:\alpha_r^n)\in\P_n(K)$ lies in the union of finitely many proper subspaces of $\P_n(K)$. One of these must contain infinitely many points $(\alpha_0^n:\dots:\alpha_r^n)$ for $n\in J$ and one of its equations can be written as $A_0x_0+\dots+A_rx_r=0$ with $A_0,\dots,A_r\in K$ not all zero.
\end{proof}

\begin{lem}\label{lemspv}
Let $\alpha$ be a pseudo-PV or a Salem number and let $\alpha=\alpha_1,\dots,\alpha_d$ be its conjugates. For all $A_0,A_1,\dots,A_d\in\overline{\Q}$, not all zero, there are only
finitely many $n \in \N$ such that $$A_0+A_1\alpha_1^n+\dots+A_d\alpha_d^n =0.$$
\end{lem}
\begin{proof}
The result is clear if $A_1=\dots=A_d=0$, since then the hypothesis is $A_0\not=0$. Otherwise, applying an automorphism $\sigma$ of $\overline{\Q}$ over $\Q$ sending some $\alpha_i$ with $A_i\not=0$ to $\alpha$, we get an equation
$$\alpha^n=-\sigma(A_0/A_i)-\sum_{\stackrel{1\le j\le d}{\scriptscriptstyle j\not=i}}\sigma(A_j/A_i)\sigma(\alpha_j)^n.$$
But the $\sigma(A_j/A_i)$ are independent of $n$ and the $\sigma(\alpha_j)$, $j\not=i$, are the conjugates of $\alpha$ distinct from $\alpha$, thus of absolute value bounded by $1$. Since $|\alpha|>1$, the absolute value $|\alpha|^n$ of the left-hand side goes to infinity with $n$ whereas that of the right-hand side remains bounded. Therefore, the equality $A_0+A_1\alpha_1^n+\dots+A_d\alpha_d^n=0$ can hold only for $n$ bounded.
\end{proof}

We now have all the ingredients to prove Theorem \ref{2}.

\newpage
\noindent
{\bf{Proof of Theorem \ref{2}.}}

\smallskip
\noindent
Let $\lambda$ and $ \alpha$ be algebraic numbers such that  $(\lambda, \alpha) \in M$.
Thus there exists $ 0 <c <1$ and infinitely many $n \in \N$ such that
$\Vert \lambda \alpha^n\Vert  < c^{n}$. If $\Vert \lambda \alpha^n\Vert  = 0$ for infinitely many $n$, then $\alpha$ is necessarily the root of an integer and also $\lambda\alpha^n \in \Z\setminus\{0\}$  for infinitely many $n$.

So we may assume that there are infinitely many $n$ such that $0 < \Vert \lambda \alpha^n\Vert  < c^{n}$. Furthermore $|\lambda\alpha^n|>1$ for $n$ large enough, since $|\alpha|>1$. Denote $J\subset\N$ the subset of $n$ for which these inequalities hold. There exists $\epsilon >0$ such that
$$ 0 < \Vert \lambda \alpha^n\Vert  <  H(\alpha^n)^{- \epsilon}$$
for $n\in J$. Thus, by Theorem \ref{cz2} with $\Gamma$ the subgroup generated by $\alpha$, $\lambda \alpha^n$ is a pseudo-PV number for all $n$ in $J$ except for a possible finite subset. Let $I\subset J$ be the infinite subset such that $\lambda\alpha^n$ is a pseudo-PV number for every $n\in I$. Since $\lambda\alpha^n$ tends to infinity whereas the absolute values of its other conjugates are bounded when $n$ tends to infinity, we may as well define the subset $I$ so that the trace of the pseudo-PV number $\lambda\alpha^n$ is non zero for $n\in I$. Then by Lemma \ref{3}, $\alpha$ is necessarily an algebraic integer.

Let $h\in\N^\times$ and $n\in I$ such that $n=mh+i$ for $0\le i<h$ and $m>\frac{r\log(H(\lambda\alpha^i))}{h\log(\alpha)}$ with $r=[\Q(\alpha^h,\lambda\alpha^i) :\Q(\alpha^h)]$. The identity map on $\Q(\alpha^h)$ gives rise to $r$ different embeddings of $\Q(\alpha^h,\lambda\alpha^i)$ over $\Q(\alpha^h)$ into  $\overline{\Q}$. But, with our condition on $m$, the corresponding $r$ embeddings of $\lambda\alpha^n$ have absolute values greater than $1$. Since $\lambda \alpha^n$ is a pseudo-PV number we deduce that $r=1$ and $\lambda\alpha^n \in \Q(\alpha^h)$ for $n\in I$. In particular, for $h=1$ we have $\lambda\in\Q(\alpha)$.

Assume a conjugate of $\alpha$ distinct from $\alpha$ has absolute value strictly greater than $1$. Write $\alpha'$ and $\lambda'\in\Q(\alpha')$ the corresponding conjugates of $\alpha$ and $\lambda\in\Q(\alpha)$. Since $\lambda\alpha^n$ is a pseudo-PV number for $n\in I$, we must have $\lambda\alpha^n=\lambda'\alpha'^n$ for $n\in I$ large enough. Forming the quotient of two such equalities for $m,n\in I$, $m<n$, we get $\alpha'^{n-m}=\alpha^{n-m}$. Let $s$ be the least common multiple of these exponents $m-n$ when $\alpha'$ runs over all the conjugates of $\alpha$ of absolute value strictly greater than $1$. Then $\alpha^s$ has exactly one conjugate of absolute value strictly greater than $1$. Since $\alpha$ is an algebraic integer,  $\alpha^s$ is either a PV or a Salem number.

Now suppose that  $s$ is an integer such that $\alpha^s$ is a Salem number. 
We  have  that there exists an $\epsilon >0 $ and  infinite set $I\subset\N^\times$ such that $\Vert\lambda\alpha^n\Vert<{\rm e}^{-\varepsilon n}$ for  all $n \in I$.
Then there exists $i\in\{0,\dots,s-1\}$ such that $$\Vert\lambda\alpha^{ms+i}\Vert<{\rm e}^{-\varepsilon (ms+i)}$$ for infinitely many $m$. By our earlier argument,
we know that $\lambda\alpha^i\in\Q(\alpha^s)$ is a rational fraction $r(\alpha^s)$ in $\alpha^s$. Let $\alpha_1,\dots,\alpha_{2d}$ be the conjugates of $\alpha^s$ of modulus $1$, that is all the conjugates except $\alpha^s$ and $\alpha^{-s}$, and $\lambda_j=r(\alpha_j)$, $j=1,\dots,2d$, be the corresponding conjugates of $\lambda\alpha^i$. If $N_m$ is the integer closest to $\lambda\alpha^{ms+i}$ and $q\in\N^\times$ a denominator of $\lambda$, we write
$$|Tr_{\Q(\alpha)/\Q}(q\lambda\alpha^{ms+i})-qN_m-q\lambda_1\alpha_1^{ms}-\dots-q\lambda_{2d}\alpha_{2d}^{ms}| < q{\rm e}^{-\varepsilon(ms+i)}+q|r(\alpha^{-s})\alpha^{-ms}|<{\rm e}^{-\varepsilon'm}$$
for some $\varepsilon'>0$ and $m$ large enough. Thus, $$\Vert q\lambda_1\alpha_1^{sm} +\dots+q\lambda_{2d}\alpha_{2d}^{sm}\Vert<{\rm e}^{-\varepsilon'm}$$ for infinitely many $m$. By Proposition~\ref{propspv}, there exists $A_0,\dots,A_{2d}\in\overline{\Q}$, not all zero, such that $$A_0+A_1\alpha_1^{sm}+\dots+A_{2d}\alpha_{2d}^{sm}=0$$ for infinitely many $m$. But by Lemma~\ref{lemspv}, this is not possible since $\alpha^s$ is a Salem number. This
completes the proof of Theorem \ref{2}.

\section{Proof of Theorems  \ref{int3real} and \ref{int3}}

For an algebraic number $\alpha $, let $k = \Q(\alpha)$ and let
$d$ be the degree of $\alpha$.
We shall need the following lemma for the proof of Theorem \ref{int3real}.

\begin{lem}\label{lemzerotracereal}
Let $\alpha$ be a real, positive algebraic number and $a,h\in\N^\times$, then the following three statements are equivalent:
\begin{itemize}
\item[1)] $Tr_{\Q(\alpha)/\Q(\alpha^h)}(\alpha^a)=0$;
\item[2)] $a$ is not divisible by $[\Q(\alpha):\Q(\alpha^h)]$.
\item[3)] $\alpha^a\notin\Q(\alpha^h)$;

\end{itemize}
\end{lem}
\begin{proof}
Set $L=\Q(\alpha^h)\subset\R$ and let $h'$ be the smallest positive integer such that $\alpha^{h'}\in L$. Obviously $1\le h'\le h$
and for any prime $p$ dividing $h'$, we have $\alpha^{h'}\notin L^p$. Otherwise $\alpha^{h'/p}$, which is the only real, positive $p$-th root of $\alpha^{h'}$, 
would belong to $L$, contradicting the minimality of $h'$. Furthermore, if $4\mid h'$ then $\alpha^{h'}\notin-4L^4$, because $-4$ is not a fourth power in $\R$. It then follows from \cite[Chap.VI, \S9, Theorem 9.1, page 297]{langalgebra}, that the polynomial 
$$X^{h'}-\alpha^{h'}$$
is irreducible in $L[X]$. Now $\alpha$ is the only real, positive root and its conjugates over $L$ are the numbers $\alpha\xi^i$, $i=0,\dots,h'-1$, where $\xi$ is a primitive $h'$-th root of unity. Thus, $h'$ is the degree of the extension $\Q(\alpha)/L$ and for all $a\in\N^\times$ we have
\begin{equation*}\label{forzerotracereal}
Tr_{\Q(\alpha)/L}( \alpha^a) = \alpha^a\sum_{i=0}^{h'-1}\xi^{ai} = \begin{cases}h'\alpha^a &\mbox{if $h'\mid a$}\cr 0 &\mbox{otherwise}\end{cases}
\end{equation*}
and $Tr_{\Q(\alpha)/L}(\alpha^a)=0$ if and only if $a\not\equiv 0\ (h')$, proving that the statements  1 and 2 are equivalent.

Now, since $\alpha^{h'}\in L$, the condition $\alpha^a\in L$ is equivalent to\footnote{Here and later $(m,n)$ denotes the gcd of $m$ and $n$.} $\alpha^{(h',a)}\in L$ and by the minimality of $h'$ this happens if and only if $(h',a)\ge h'$, that is $h'\mid a$. This shows that the statements 2 and 3 are equivalent, ending the proof of the lemma.
\end{proof}

\noindent
{\bf Proof of Theorem \ref{int3real}.}
\smallskip

Let $H$ be the splitting field of the minimal polynomial of $\alpha$ over $\Q$. So  $h$ is the order of the torsion group of $H^\times$. Set $L=\Q(\alpha^h)$ and $K=L(\alpha)=\Q(\alpha)\subset H$. 

In one direction, assume that $Tr_{K/\Q}(\alpha^n)=0$ for infinitely many $n\in\N$. Since $\alpha\in K$ we have $Tr_{H/K}(\alpha^n)=[H:K]\alpha^n$ and, since $\N$ is the disjoint union of the congruence classes $a+h\N$ for $a=0,\dots,h-1$, we deduce from the hypothesis that
$$
Tr_{H/\Q}(\alpha^{a+hm}) = Tr_{K/\Q}\left(Tr_{H/K}(\alpha^{a+hm})\right) = [H:K]Tr_{K/\Q}(\alpha^{a+hm}) = 0
$$
for infinitely many $m\in\N$ and some $0\le a<h$. Let $d=[L:\Q]$ and $\tau_1,\dots,\tau_d$ be all the embeddings of $L$ in $H$ over $\Q$. We express $Tr_{H/\Q}(\alpha^{a+hm})$ as the $m$-th term of a linear recurrence sequence:
$$
Tr_{H/\Q}( \alpha^{a+hm}) = Tr_{L/\Q}\left(Tr_{H/L}(\alpha^{a+hm})\right) = \sum_{i=1}^d\tau_i\left(Tr_{H/L}(\alpha^a)\right)\tau_i(\alpha^h)^m,
$$
which vanishes for infinitely many $m\in\N$. Observe that the ratios $\tau_i(\alpha^h)/\tau_j(\alpha^h)$, $i\not=j$, are not roots of unity, because being roots of unity and $h$ power in $H$ they would be $1$ but $\tau_i\not=\tau_j$ on $L$ and we must have $\tau_i(\alpha^h)\not=\tau_j(\alpha^h)$. The Skolem-Mahler-Lech's theorem \cite[Corollary 7.2, page 193]{cetraro} implies that $Tr_{H/L}(\alpha^a)=0$ and then $Tr_{K/L}(\alpha^a)=0$ because
 $$
Tr_{H/L}(\alpha^a) = Tr_{K/L}\left(Tr_{H/K}(\alpha^a)\right) = [H:K]Tr_{K/L}(\alpha^a).
 $$
But then Lemma~\ref{lemzerotracereal} ensures that $\alpha^a\notin\Q(\alpha^h)$ and hence in particular $\alpha\notin\Q(\alpha^h)$.

In the other direction, the same Lemma~\ref{lemzerotracereal} shows that if $\alpha\notin\Q(\alpha^h)$, then $Tr_{K/L}(\alpha)=0$ and
$$Tr_{K/\Q}(\alpha^{1+hm})=Tr_{L/\Q}\left(\alpha^{hm}Tr_{K/L}(\alpha)\right)=0$$
for all $m\in\N$. This completes the proof.

\begin{rmk}
The condition $\alpha\notin\Q(\alpha^h)$ is also equivalent to $\alpha^h\notin\Q(\alpha^h)^h$.
\end{rmk}

We now prove the following variant of Lemma \ref{lemzerotracereal} which we shall require
a little later. 

\begin{lem}\label{lemzerotrace}
Let $\alpha$ be a nonzero algebraic number, $a,h\in\N^\times$ and $\zeta$ a primitive $h$-th root of unity. 
Then $\alpha^a\notin\Q(\alpha^h,\zeta)$ if and only if $Tr_{\Q(\alpha,\zeta)/\Q(\alpha^h,\zeta)}(\alpha^a)=0$.

Furthermore, $Tr_{\Q(\alpha)/\Q}(\alpha^a)=0$ for all integer $a$ not divisible by $[\Q(\alpha,\zeta):\Q(\alpha^h,\zeta)]$.
\end{lem}
\begin{proof}
Let $L=\Q(\alpha^h,\zeta)$ and $K=L(\alpha)$, since $\alpha$ is a root of the polynomial 
$$X^{h}-\alpha^h\in L[X]$$
the extension $K/L$ is cyclic of degree $h'\mid h$, the conjugates of $\alpha$ over $L$ are $\alpha\xi^i$, $i=0,\dots,h'-1$, where $\xi$ is a primitive $h'$-th root of unity, and $\alpha^{h'}\in L$, {\it see} \cite[Chap.VI, \S6, Theorem 6.2, page 289]{langalgebra}. Thus, for all $a\in\N^\times$ we have
\begin{equation}\label{forzerotrace}
Tr_{K/L}( \alpha^a) = \alpha^a\sum_{i=0}^{h'-1}\xi^{ai} = \begin{cases}h'\alpha^a &\mbox{if $h'\mid a$}\cr 0 &\mbox{otherwise}\end{cases}
\end{equation}
and $Tr_{K/L}(\alpha^a)=0$ if and only if $a\not\equiv0\ (h')$. We observe that $h'$ is the smallest positive integer such that $\alpha^{h'}\in L$, otherwise another application of {\it ibidem} would lead to a contradiction on the degree of the extension $K/L$. Now, since $\alpha^{h'}\in L$, the condition $\alpha^a\in L$ is equivalent to $\alpha^{(h',a)}\in L$ and by the minimality of $h'$ this happens if and only if $(h',a)\ge h'$, that is $h'\mid a$.

For the last statement, if $a$ is not divisible by $h'$ then
$$Tr_{K/\Q}(\alpha^{a})=Tr_{L/\Q}\left(Tr_{K/L}(\alpha^{a})\right)=0,$$
because $Tr_{K/L}( \alpha^a)=0$ by \eqref{forzerotrace}.
But, $Tr_{K/\Q(\alpha)}(\alpha^a)=[K:\Q(\alpha)]\alpha^a$ and we can further write
$$Tr_{\Q(\alpha)/\Q}(\alpha^{a})=[K:\Q(\alpha)]^{-1}Tr_{K/\Q}(\alpha^{a})=0
.$$
\end{proof}

We end with the proof  of Theorem~\ref{int3} which works
for general algebraic numbers.

\medskip
\noindent
{\bf Proof of Theorem \ref{int3}}.

\smallskip
Let $H$ be the splitting field of the minimal polynomial of $\alpha$ over $\Q$. So that $h$ is the order of the torsion group of $H^\times$ and $\zeta$ is a primitive $h$-th root of unity. Set $H_0=\Q(\zeta)$, $L=H_0(\alpha^h)$ and $K=L(\alpha)\subset H$. 

In one direction, assume $Tr_{K/H_0}(\alpha^n)=0$ for infinitely many $n\in\N$. Since $\alpha\in K$ we have $Tr_{H/K}(\alpha^n)=[H:K]\alpha^n$ and we deduce from the hypothesis that
$$
Tr_{H/H_0}(\alpha^{a+mh}) = Tr_{K/H_0}\left(Tr_{H/K}(\alpha^{a+mh})\right) = [H:K]Tr_{K/H_0}(\alpha^{a+mh}) = 0
$$
for infinitely many $m\in\N$ and some $0\le a<h$. Let $d=[L:H_0]$ and $\tau_1,\dots,\tau_d$ be all the embeddings of $L$ in $H$ over $H_0$. We express $Tr_{H/H_0}(\alpha^{a+mh})$ as the $m$-th term of a linear recurrence sequence:
$$
Tr_{H/H_0}( \alpha^{a+mh}) = Tr_{L/H_0}\left(Tr_{H/L}(\alpha^{a+mh})\right) = \sum_{i=1}^d\tau_i\left(Tr_{H/L}(\alpha^a)\right)\tau_i(\alpha^h)^m,
$$
which vanishes for infinitely many $m\in\N$. Observe that the ratios $\tau_i(\alpha^h)/\tau_j(\alpha^h)$, $i\not=j$, are not roots of unity, because being roots of unity and $h$ power in $H$ they would be $1$, but $\tau_i\not=\tau_j$ on $L$ and, since $\tau_i=\tau_j$ on $H_0$, we must have $\tau_i(\alpha^h)\not=\tau_j(\alpha^h)$. The Skolem-Mahler-Lech's theorem \cite[Corollary 7.2, page 193]{cetraro} implies that $Tr_{H/L}(\alpha^a)=0$ and then $Tr_{K/L}(\alpha^a)=0$ because
 $$
Tr_{H/L}(\alpha^a) = Tr_{K/L}\left(Tr_{H/K}(\alpha^a)\right) = [H:K]Tr_{K/L}(\alpha^a).
 $$
As before, Lemma~\ref{lemzerotrace}  ensures that $\alpha^a\notin\Q(\alpha^h,\zeta)$ and, in particular, $\alpha\notin\Q(\alpha^h,\zeta)$.

In the other direction, Lemma~\ref{lemzerotrace} also shows that if $\alpha^a\notin\Q(\alpha^h,\zeta)$, then $Tr_{K/L}(\alpha^{a})=0$ and
$$Tr_{K/H_0}(\alpha^{a+mh})=Tr_{L/H_0}\left(\alpha^{mh}Tr_{K/L}(\alpha^{a})\right)=0$$
for all $m\in\N$. This completes the proof.

\section{Periodicity and description of the set $H\cap\overline{\Q}^2$}
Let $\alpha$ be an algebraic integer, $P_\alpha\in\Z[X]$ its minimal monic polynomial over $\Z$ and set $k=\Q(\alpha)$. Recall from \cite[Chap.III, \S1, Cor. to Prop.2]{lang} that $\frac{1}{P'_\alpha(\alpha)}\Z[\alpha]$ is the complementary module $\Z[\alpha]'$ of $\Z[\alpha]$, that is the set of elements $y\in\Q(\alpha)$ such that $Tr_{k/\Q}(y\Z[\alpha])\subset\Z$. Indeed, without the assumption made in the above cited Corollary, its proof shows the desired equality.  Thus for $\lambda\in\frac{1}{P'_\alpha(\alpha)}\Z\left[\alpha,\frac{1}{\alpha}\right]$ and $n$ large enough we have $Tr_{k/\Q}(\lambda\alpha^n)\in\Z$.

The following lemma can also be found in \cite[Lemma 2]{dubickas} and \cite[Lemma 2]{zaimi}.
\begin{lem}\label{lemper}
Let $\alpha$ be an algebraic integer of degree $d$ over $\Q$, $\lambda\in\Z[\alpha]'$ and $b\in\N^\times$, then the sequence $(Tr_{k/\Q}(\lambda\alpha^n) \mod b)_{n\in\N}$ is ultimately periodic and the period has length $\leq b^d$.
\end{lem}
\begin{proof}
Let $t_n\in\Z/b\Z$ be the class of $Tr_{k/\Q}(\lambda\alpha^n)$ modulo $b$ and write  
$$P_\alpha(X)=X^d+A_{d-1}X^{d-1}+\dots+A_0\in\Z[X]$$
the minimal monic polynomial of $\alpha$ over $\Z$. Let $\overline{A_i}$  denote the class of $A_i$ modulo $b$. From $Tr_{k/\Q}(P_\alpha(\alpha)\lambda\alpha^n)=0$ we deduce
\begin{equation}\label{forinlemper}
t_{n+d}=-\overline{A_{d-1}}t_{n+d-1}-\dots-\overline{A_0}t_n
\end{equation}
for all $n\in\N$. Since there are finitely many $d$-tuples of elements of $\Z/b\Z$, at least one must appear twice as blocks in the sequence $(t_n)_{n\in\N}$. Thus there exists natural numbers $m<n$ such that $t_m=t_n$, \dots, $t_{m+d-1}=t_{n+d-1}$.
It follows inductively from \eqref{forinlemper} that $$t_{m+d}=t_{n+d}, t_{m+d+1}=t_{n+d+1}, \dots\ $$ and hence the ultimate periodicity of the sequence. The length of the period divides $n-m$ and since there are at most $b^d$ distinct blocks of $d$ elements in $\Z/b\Z$, we have  that the period has length  at most $b^d$.
\end{proof}

\begin{rmk}\label{rmksalem}
If the number $\alpha$ is assumed to be a unit, then equation \eqref{forinlemper} enables one to deduce that the sequence $(t_n)_{n\in\N}$ is purely periodic and in particular, if $p$ denotes the length of the period, that $t_{pn}$ is the class of $Tr_{k/\Q}(\lambda)$ modulo $b$ for all $n\in\N$, see \cite[Lemma 2]{zaimi}.
\end{rmk}

\begin{prop}\label{proper}
Let $\alpha$ be a PV number of degree $d$ over $\Q$, $\lambda_0\in\Z[\alpha]'$ and $b\in\N^\times$. There exists an integer $p\leq b^d$ and\footnote{Here and after $\lceil*\rceil$ stands for the least integer larger or equal to $*$ and $\lfloor*\rfloor$ is the largest integer smaller or equal to $*$ (the integer part).} $i_1,\dots,i_p\in\{-\lceil b/2\rceil+1,\dots,\lfloor b/2\rfloor\}$ such that $Tr_{k/\Q}(\lambda_0\alpha^n)\equiv i_\ell(b)$ and $\left|\left\Vert\frac{\lambda_0\alpha^n}{b}\right\Vert-\frac{|i_{\ell}|}{b}\right|<c^n$ with $\ell=n-p\lfloor n/p\rfloor+1$, $n$ large enough and some $0<c<1$.
\end{prop}
\begin{proof}
By Lemma \ref{lemper}, the sequence of classes $t_n\in\Z/b\Z$ modulo $b$ of $Tr_{k/\Q}(\lambda_0\alpha^n)$ is ultimately periodic, with period of length say $p$. We represent the elements of $\Z/b\Z$ by the integers $-\lceil b/2\rceil+1,\dots,\lfloor b/2\rfloor$ so that the period of the sequence $(t_n)_{n\in\N}$ gives integers $i_1,\dots,i_p$ lying in
$\{-\lceil b/2\rceil+1,\dots,\lfloor b/2\rfloor\}$ satisfying $Tr_{k/\Q}(\lambda_0\alpha^n)\equiv i_\ell(b)$ for $\ell=n-p\lfloor n/p\rfloor+1$ and $n$ large. We then observe
$$\left|\frac{\lambda_0\alpha^n}{b} - \frac{Tr_{k/\Q}(\lambda_0\alpha^n)-i_\ell}{b} - \frac{i_\ell}{b}\right| = \left|\frac{\lambda_0\alpha^n}{b} - Tr_{k/\Q}\left(\frac{\lambda_0\alpha^n}{b}\right)\right| < c^n$$
for some $0<c<1$, because $\alpha$ is a PV number. It follows that $\frac{Tr_{k/\Q}(\lambda_0\alpha^n)-i_\ell}{b}$ is the integer closest to $\frac{\lambda_0\alpha^n}{b}$ and $\left\Vert\frac{\lambda_0\alpha^n}{b}\right\Vert=\frac{|i_\ell|}{b}+O(c^n)$.
\end{proof}

\begin{rmk}
Observe that $i_1$ is the residue modulo $b$ of $Tr_{k/\Q}(\lambda_0\alpha^n)$ for all large $n$ divisible by $p$. For any $j\in\Z$, replacing $\lambda_0$ by $\lambda_0\alpha^j$
simply shifts the sequence $(Tr_{k/\Q}(\lambda_0\alpha^n))_{n\in\N}$ by $|j|$ steps to the left or right according to the sign of $j$. Thus, with the PV number $\alpha$ and the integer $b\in\N^\times$ fixed, Proposition \ref{proper} associates to each $\lambda_0\in\frac{1}{P'_\alpha(\alpha)}\Z\left[\alpha,\frac{1}{\alpha}\right]$, the integer $p$ and the vector $(i_1,\dots,i_p)\in(\Z/b\Z)^p$. This map is a homomorphism of $\Z\left[\alpha,\frac{1}{\alpha}\right]$-modules if we define the action of $\alpha$ on $(\Z/b\Z)^p$ as the cyclic permutation $\alpha\cdot(i_1,\dots,i_p)=(i_2,\dots,i_p,i_1)$. In particular, the integer $p$ can be chosen independent of $\lambda_0$, although for some $\lambda_0$ a shorter period may exist.
\end{rmk}

More generally for  any $\lambda\in\Q(\alpha)$, we choose  $b$ to be the least positive integer such that $\lambda_0=b\lambda\in\frac{1}{P'_\alpha(\alpha)}\Z\left[\alpha,\frac{1}{\alpha}\right]$ and consider the integer $p$ and vector $(i_1,\dots,i_p)\in(\Z/b\Z)^p$ associated to $b\lambda$. It follows from Proposition \ref{proper} that the fractions $\frac{|i_1|}{b},\dots,\frac{|i_p|}{b}$ (lying in $[0,\frac{1}{2}]$) are the limit points of the sequence $(\Vert\lambda\alpha^n\Vert)_{n\in\N}$.

\begin{cor}\label{corperM}
Let $\alpha$ be a PV number and $\lambda\in\Q(\alpha)$. Then, $0$ is a limit point of the sequence $\left(\left\Vert\lambda\alpha^n\right\Vert\right)_{n\in\N}$ if and only if there exists an integer $p\in\N^\times$ and $\ell\in\{1,\dots,p\}$ such that $Tr_{k/\Q}(\lambda\alpha^{np+\ell-1})\in\Z$  for $n\in\N$ large enough or, equivalently, $\lambda$ belongs to $\frac{1}{\alpha^{\ell-1} P'_{\alpha^p}(\alpha^p)}\Z\left[\alpha^p,\frac{1}{\alpha^p}\right]$.

And $0$ is the unique limit point of the sequence  $\left(\left\Vert\lambda\alpha^n\right\Vert\right)_{n\in\N}$ if and only if $\lambda\in\frac{1}{P'_{\alpha}(\alpha)}\Z\left[\alpha,\frac{1}{\alpha}\right]$.
\end{cor}
\begin{proof}
Assume $0$ is a limit point of the sequence $\left(\left\Vert\lambda\alpha^n\right\Vert\right)_{n\in\N}$. Write $\lambda=\frac{1}{b}f(\alpha)$ with $b\in\N^\times$ and $f\in\Z[X]$. Set $\lambda_0=b\lambda=f(\alpha)$. Since $0$ is a limit point of the sequence $\left(\left\Vert\frac{f(\alpha)\alpha^n}{b}\right\Vert\right)_{n\in\N}$, it follows that in Proposition \ref{proper} some $i_\ell$ must be $0$ and the numbers $Tr_{k/\Q}(f(\alpha)\alpha^{np+\ell-1})$ are divisible by $b$ or, equivalently, $Tr_{k/\Q}(\lambda\alpha^{np+\ell-1})\in\Z$, for $n$ large enough. 
Thus $\lambda\alpha^{n_0p+\ell-1}\in\frac{1}{P'_{\alpha^p}(\alpha^p)}\Z[\alpha^p]$ for some $n_0\in\N$.

Conversely assume $Tr_{k/\Q}(\lambda\alpha^{np+\ell-1})\in\Z$. Since $\alpha$ is a PV number and $n$ is large enough, $\lambda\alpha^{np+\ell-1}$ is a large real number whereas the sum of its conjugates is of absolute values $<c^n$ for some $0<c<1$. The difference $\Vert\lambda\alpha^{np+\ell-1}\Vert=|\lambda\alpha^{np+\ell-1}-Tr_{k/\Q}(\lambda\alpha^{np+\ell-1})|<c^n$ tends to $0$ as $n$ goes to $\infty$ and thus $0$ is a limit point of the sequence $\left(\left\Vert\lambda\alpha^n\right\Vert\right)_{n\in\N}$.

If $0$ is the unique limit point of the sequence  $\left(\left\Vert\lambda\alpha^n\right\Vert\right)_{n\in\N}$, we must have $p=1$ and $i_1=0$ in Proposition \ref{proper}. Thus $\ell=1$ and $\lambda\alpha^{n_0}$ belongs to the complementary module of $\Z[\alpha]$, for some $n_0$. Conversely if $\lambda\in\frac{1}{P'_{\alpha}(\alpha)}\Z\left[\alpha,\frac{1}{\alpha}\right]$ (i.e. $p=\ell=1$), then $Tr_{k/\Q}(\lambda\alpha^n)\in\Z$ for all $n$ large and the above argument shows that it is the closest integer to $\lambda\alpha^n$, since $\alpha$ is a PV number. Then, $\Vert\lambda\alpha^n\Vert = |\lambda\alpha^{n}-Tr_{k/\Q}(\lambda\alpha^{n})|<c^n$ tends to $0$ as $n$ goes to $\infty$.
\end{proof}

\begin{cor}\label{corper}
Let $(\lambda,\alpha)\in H\cap(\R^\times\times\overline{\Q})$, then $\alpha$ is a PV number, $\lambda\in\Q(\alpha)$ and $Tr_{k/\Q}(\lambda\alpha^n)\in\Z$ for $n$ large enough.
\end{cor}
\begin{proof}
For $(\lambda,\alpha)\in H\cap(\R^\times\times\overline{\Q})$, we know by Theorem \ref{hardy} that $\alpha$ is a PV number and $\lambda\in\Q(\alpha)$. We may write $\lambda=\frac{1}{b}f(\alpha)$ for some rational integer $b$ and $f\in\Z[X]$. By Proposition \ref{proper}, the limit points of the sequence $\left(\left\Vert\frac{f(\alpha)\alpha^n}{b}\right\Vert\right)_{n\in\N}$ are those rational numbers among $0,1/b,\dots,\lfloor b/2\rfloor/b$ for which the numerator is congruent to infinitely many $Tr_{k/\Q}(f(\alpha)\alpha^n)$ modulo $b$. But, $(\lambda,\alpha)$ being in $H$, the sequence converges to $0$. Thus, the only possible fraction is $0$ and all the numbers $Tr_{k/\Q}(f(\alpha)\alpha^n)$ must be divisible by $b$ and so $Tr_{k/\Q}(\lambda\alpha^n)\in\Z$, for $n$ large enough.
\end{proof}

We need to introduce the following notation for the purpose of the next proposition. For any set $S\subset\R$ we define $\Vert S\Vert=\{\Vert\gamma\Vert;\gamma\in S\}\subset[0,1/2]$.

\begin{prop}\label{propersalem}
Let $\alpha$ be a Salem number of degree $d$ over $\Q$, $\lambda\in\Z[\alpha]'$ and $b\in\N^\times$. There exists an integer $p\leq b^d$ and $i_1,\dots,i_p\in\{-\lceil b/2\rceil+1,\dots,\lfloor b/2\rfloor\}$ such that $Tr_{k/\Q}(\lambda\alpha^n)\equiv i_\ell(b)$ for $\ell=n-p\lfloor n/p\rfloor+1$ and any $n\in\N^\times$. Furthermore, let $\alpha=\alpha_1,\alpha^{-1}=\alpha_2,\alpha_3,\dots,\alpha_{d}$ be the conjugates of $\alpha$ and $\lambda_{\ell,1},\dots,\lambda_{\ell,d}$ be those of $\lambda\alpha^{\ell-1}$ for $\ell\in\{1,\dots,p\}$. Then the set of limit points of each subsequence $\left(\left\Vert\frac{\lambda}{b}\alpha^{mp+\ell-1}\right\Vert\right)_{m\in\N}$ is $\Vert\mathcal{L}_\ell\Vert$, where
$$\mathcal{L}_\ell=\left[\frac{i_\ell}{b}-\sum_{j=3}^d\frac{|\lambda_{\ell,j}|}{b}, \frac{i_\ell}{b}+\sum_{j=3}^d\frac{|\lambda_{\ell,j}|}{b}\right].$$

Furthermore, an integer closest to $\lambda\alpha^n$ is congruent to $i\in\{-\lceil b/2\rceil+1,\dots,\lfloor b/2\rfloor\}$ modulo $b$ for infinitely many $n$ if and only if  we have $\left\Vert\frac{\lambda\alpha^n-i}{b}\right\Vert\in \left[0,\frac{1}{2b}\right]=\left\Vert\left[-\frac{1}{2b},\frac{1}{2b}\right]\right\Vert$ or equivalently $\left\Vert\frac{\lambda\alpha^n}{b}\right\Vert\in\left\Vert\left[\frac{2i-1}{2b},\frac{2i+1}{2b}\right]\right\Vert$, for infinitely many $n$. In this case $\left(\cup_{\ell=1}^{p}\mathcal{L}_\ell\right)\cap\left\Vert\left[\frac{2i-1}{2b},\frac{2i+1}{2b}\right]\right\Vert\not=\emptyset$.
\end{prop}
\begin{proof}
The first part of the statement follows from Lemma~\ref{lemper} and Remark~\ref{rmksalem}, which asserts that 
the sequence of classes $t_n\in\Z/b\Z$ modulo $b$ of $Tr_{k/\Q}(\lambda\alpha^n)$ is purely periodic, with period of length say $p$. 
We represent the elements of $\Z/b\Z$ by the integers $-\lceil b/2\rceil+1,\dots,\lfloor b/2\rfloor$ so that the period of the sequence $(t_n)_{n\in\N}$ gives $i_1,\dots,i_p\in\{-\lceil b/2\rceil+1,\dots,\lfloor b/2\rfloor\}$ satisfying $Tr_{k/\Q}(\lambda\alpha^n)\equiv i_\ell(b)$ for $\ell=n-p\lfloor n/p\rfloor+1$ and $n$ large. We have
\begin{equation}\label{dspropersalem}
\lambda\alpha^{mp+\ell-1} = Tr_{k/\Q}(\lambda\alpha^{mp+\ell-1}) - \sum_{j=3}^{d}\lambda_{\ell,j}\alpha_j^{mp} - \lambda_{\ell,2}\alpha_2^{mp}.
\end{equation}

Let $\ell\in\{1,\dots,p\}$, $\rho\in\Vert\mathcal{L}_\ell\Vert$ and $\gamma\in\mathcal{L}_\ell$ such that $\rho=\Vert\gamma\Vert$. Recall that the conjugates of $\alpha$ other than $\alpha$ and $1/\alpha$ are of the form ${\rm e}^{\pm2\pi{\rm i}\theta_j}$, $j=1,\dots,(d-2)/2$. Since $1,\theta_1,\dots,\theta_{(d-2)/2}$ are linearly independent over $\Z$ (see \cite[page 32]{salem}), it follows from Kronecker's theorem (\cite[Appendix 8]{salem}, see also~\cite[Lemma 1]{zaimi}) that there exists an infinite subset $I\subset\N$ such that the sequence $\left(\sum_{j=3}^d\lambda_{\ell,j}\alpha_j^{mp}\right)_{m\in I}$ converges to $i_\ell-b\gamma$. Since $\lambda_{\ell,2}\alpha_2^{mp}$ tends to $0$ as $m$ goes to $\infty$ and $Tr_{k/\Q}(\lambda\alpha^{mp+\ell-1})\equiv i_\ell(b)$, we deduce from~\eqref{dspropersalem}
$$
\left\Vert\frac{\lambda}{b}\alpha^{mp+\ell-1}\right\Vert = \left\Vert\frac{i_\ell}{b} - \sum_{j=3}^d\frac{\lambda_{\ell,j}}{b}\alpha_j^{mp} -\frac{\lambda_{\ell,2}}{b}\alpha_2^{mp} \right\Vert = \left\Vert\gamma+\varepsilon_m \right\Vert 
$$
for $m\in I$ and where $(\varepsilon_m)_{m\in I}$ converges to $0$ as $m\in I$ goes to $\infty$. Thus the limit point of $\left(\Vert\gamma+\varepsilon_m\Vert\right)_{m\in I}$ is $\Vert\gamma\Vert=\rho$.

Conversely, by~\eqref{dspropersalem} each term of the series $\left(\left\Vert\frac{\lambda}{b}\alpha^{mp+\ell-1}\right\Vert\right)_{m\in\N}$ can be rewritten $\left(\left\Vert\gamma_m\right\Vert\right)_{m\in\N}$ with $\gamma_m=\frac{i_\ell}{b}-\sum_{j=3}^d\frac{\lambda_{\ell,j}}{b}\alpha_j^{mp}-\frac{\lambda_{\ell,2}}{b}\alpha_2^{mp}\in\R$, which satisfies $|b\gamma_m-i_\ell|\le\sum_{j=3}^d|\lambda_{\ell,j}|+|\lambda_{\ell,2}\alpha_2^{mp}|$. The integer closest to $\gamma_m$ can take only finitely many values. Thus, if a subsequence $\left(\left\Vert\gamma_{\smash{m}}\right\Vert\right)_{m\in I}$ ($I\subset\N$ infinite) converges to a limit $\rho$, some subsequence $\left(\gamma_{\smash{m}}\right)_{m\in J}$ ($J\subset I$ infinite) converges to a limit $\gamma$. This limit $\gamma$ satisfies $\Vert\gamma\Vert=\rho$ and $|b\gamma-i_\ell|\le\sum_{j=3}^d|\lambda_{\ell,j}|$, which shows $\rho\in\Vert\mathcal{L}_\ell\Vert$.

If an integer closest to $\lambda\alpha^n$ is written $bN_n+i$, then $\left|\frac{\lambda\alpha^n-i}{b}-N_n\right|\le\frac{1}{2b}$ which entails $\left\Vert\frac{\lambda\alpha^n-i}{b}\right\Vert\le\frac{1}{2b}$. In the other direction, if the latter inequality holds, there exists an integer $N_n$ such that $|\lambda\alpha^n-i-bN_n|\le \frac{1}{2}$ which shows that $bN_n+i$ is an integer closest to $\lambda\alpha^n$. If $\left\Vert\frac{\lambda\alpha^n}{b}\right\Vert\in\left\Vert\left[\frac{2i-1}{2b},\frac{2i+1}{2b}\right]\right\Vert$ for infinitely many $n$, the sequence $\left(\left\Vert\frac{\lambda\alpha^n}{b}\right\Vert\right)_{n\in\N}$ has a limit point in $\left\Vert\left[\frac{2i-1}{2b},\frac{2i+1}{2b}\right]\right\Vert$ which also belongs to $\cup_{\ell=1}^{p}\Vert\mathcal{L}_\ell\Vert$.
%
\end{proof}

The above result naturally leads to the following problem:

\medskip
\noindent
{\bf Problem}: For a Salem number $\alpha$, characterise the algebraic numbers $\lambda \in \Q(\alpha)$ such that the sequence ($\Vert \lambda\alpha^n\Vert $) is dense in $[0,1/2]$.

\begin{rmk}
The distribution of the sequence of fractional parts of $\lambda\alpha^n/b$ in $[0,1]$ is obtained from that of the difference with the nearest integer, by exchanging the subintervals $[-1/2,0]$ and $[0,1/2]$, while the sequence $(\Vert\alpha^n/b\Vert)_{n\in\N}$ is the superposition of these two latter intervals head to tail.

We remind the reader that for a Salem number $\alpha$, the sequence $(\Vert\alpha^n\Vert)_{n\in\N}$ is dense, but not uniformly distributed in $[0,1/2]$. In fact, when $\lambda=1$ and $b\leq 2d-4$ the length of each interval $\mathcal{L}_\ell$ is at least $1$ and thus the sequence of fractional parts of $\alpha^n/b$ is dense in $[0,1]$, as in \cite[Theorem (i)]{zaimi}. Further, Proposition~\ref{propersalem} together with Remark~\ref{rmksalem} also allows us to recover \cite[Theorem (ii) and (iii)]{zaimi}. 

However, when $b>2d-4$ the behaviour of the sequence $(\Vert\alpha^n/b\Vert)_{n\in\N}$ (as well as the sequence of fractional parts) strongly depends on the period $i_1,\dots,i_p$, which is somewhat mysterious. We give below two contrasting examples illustrating this phenomenon.
\end{rmk}
\vskip-2mm

\begin{figure}[ht]
\centerline{\hskip-1cm\includegraphics[width=10cm]{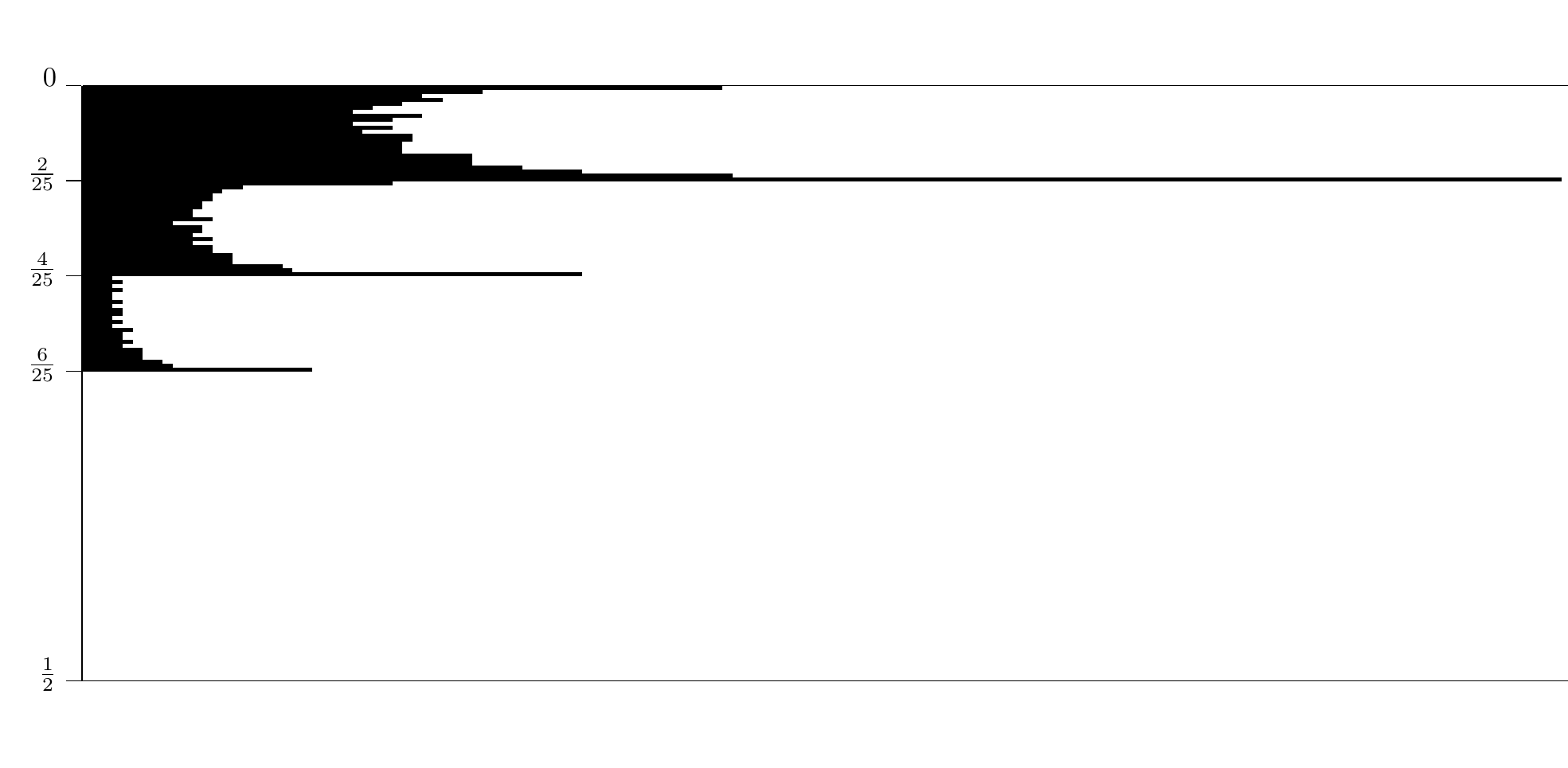}}\vskip-8mm
\caption{Scattering of $\Vert\lambda\alpha^n\Vert$ for $\lambda=1/25$ and $\alpha$ a Salem number, $n\le2500$.}
\label{fig:salemi}
\end{figure}

\begin{figure}[ht]
\centerline{\hskip-1cm\includegraphics[width=10cm]{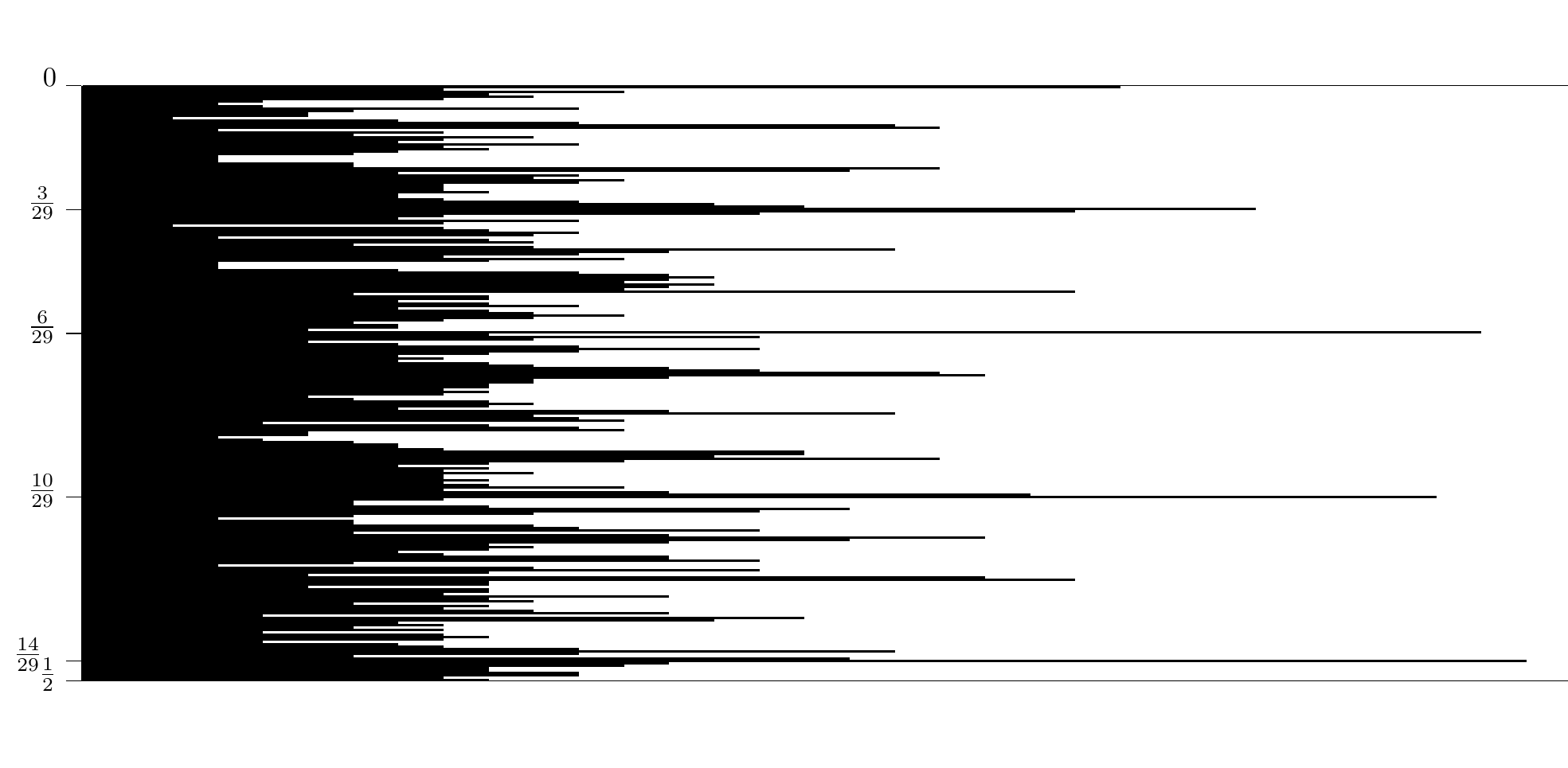}}\vskip-8mm
\caption{Scattering of $\Vert\lambda\alpha^n\Vert$ for $\lambda=1/29$ and $\alpha$ a Salem number, $n\le2500$.}
\label{fig:salemii}
\end{figure}

\begin{exmpl}
Figure~\ref{fig:salemi} shows the distribution of the numbers $\Vert\lambda\alpha^n\Vert$ in the interval $[0,1/2]$ for $n\le2500$, where $\alpha$ is the Salem number root of the polynomial $x^4-25x^3+x^2-25x+1$ and $\lambda=1/25$. The period is $0,-2,0,-2,0,4$ of length $p=6$. This gives the three intervals $[0,2/25]$, $[0,4/25]$ and $[2/25,6/25]$, which cover the whole set of limit points $[0,6/25]$, with special concentrations on the four values $0$, $2/25$, $4/25$ and $6/25$.
\end{exmpl}
\vskip-2mm

Similar pictures are obtained for $\lambda=1/24$ and $1/26$, but in general the distribution of the numbers $\Vert\lambda\alpha^n\Vert$ is dense in $[0,1/2]$ (although not uniformly), as in the next example.

\begin{exmpl}
Figure~\ref{fig:salemii} shows the distribution of the numbers $\Vert\lambda\alpha^n\Vert$ in the interval $[0,1/2]$ for $n\le2500$, where $\alpha$ is again the Salem number root of the polynomial $x^4-25x^3+x^2-25x+1$ and $\lambda=1/29$. The period involves all the 29 integers between $-14$ and $14$. This gives $15$ intervals $[0,2/29]$, $[0,3/29]$, $[(i-2)/29,(i+2)/29]$, $i=2,\dots,14$, which cover the whole interval $[0,1/2]$, with several special concentrations.
\end{exmpl}
We refer to the interested reader the papers {\cite{aki} and \cite{doche}
where other aspects of distribution of powers of
Salem numbers is studied. See also the papers \cite{bugeaud} and \cite{see}
where the set of limit points of the sequences $\{\Vert \alpha^n \Vert^{\frac{1}{n}}\}$
is investigated.

\medskip
We will now use these periodicity properties to prove Theorem \ref{hardy+}
which refines the conclusion of Hardy's Theorem \ref{hardy}.

\medskip
\noindent
{\bf Proof of Theorem \ref{hardy+}.}
\medskip

When $\alpha$ is a PV number and $n$ is large enough, then $\lambda\alpha^n$ is a large real number whereas the sum of its conjugates is of absolute values $<c^n$ for some $0<c<1$. Hence we have that  the difference $|\lambda\alpha^n-Tr_{k/\Q}(\lambda\alpha^n)|<c^n$ tends to $0$ as $n \to \infty$. 
Further, if $\lambda\alpha^n$ belongs to the complementary module of $\Z[\alpha]$ for $n$ large enough, then $Tr_{k/\Q}(\lambda\alpha^n)\in\Z$ and 
hence $(\lambda,\alpha)\in H\cap\overline{\Q}^2$.

In the other direction, it follows from Corollary \ref{corper} that if $(\lambda,\alpha)\in H\cap\overline{\Q}^2$, then  $\alpha$ is a PV number, $\lambda\in\Q(\alpha)$ and $\lambda\alpha^n$ belongs to the complementary module of $\Z[\alpha]$ for $n$ large enough. This latter condition can be rewritten $\lambda\in\Z[\alpha]'\left[\frac{1}{\alpha}\right]=\frac{1}{P'_\alpha(\alpha)}\Z\left[\alpha,\frac{1}{\alpha}\right]$.$\phantom{mannnnnnmma} \Box$

\bigskip

We now give the proof of Theorem \ref{mahler+} which is an analogus result for Mahler sets.

\medskip
\noindent
{\bf Proof of Theorem \ref{mahler+}.}
\medskip

If $(\lambda,\alpha)\in M\cap\overline{\Q}^2$, then $\Vert\lambda\alpha^n\Vert<c^n$ for infinitely many $n$. By Theorem \ref{2}, we know that $\alpha^{s_0}$ is a PV number for some $s_0\in\N^\times$. Also because $(\lambda,\alpha)\in M$, there exists an integer $0\le i<s_0$ such that $0$ is a limit point of the sequence $\left(\Vert\lambda\alpha^{ns_0+i}\Vert\right)_{n\in\N}$. It follows from Corollary \ref{corperM} that there exists integers $1\le \ell\le p$ such that $\lambda\alpha^i\in\frac{1}{\alpha^{s_0(\ell-1)}P'_{\alpha^{ps_0}}(\alpha^{ps_0})}\Z\left[\alpha^{ps_0},\frac{1}{\alpha^{ps_0}}\right]$. Setting $s=ps_0$ and $t=s_0(\ell-1)+i<\ell s_0\le ps_0=s$ proves the assertion, because $\alpha^{s}=\alpha^{ps_0}$ is again a PV number.

Conversely, if $\lambda\alpha^{t}\in\frac{1}{P'_{\alpha^{s}}(\alpha^{s})}\Z\left[\alpha^{s},\frac{1}{\alpha^{s}}\right]$, then for $n$ large enough $Tr_{k/\Q}(\lambda\alpha^{ns+t})\in\Z$. But, since $\alpha^s$ is a PV number, $\Vert\lambda\alpha^{ns+t}\Vert=\left|\lambda\alpha^{ns+t}-Tr_{k/\Q}(\lambda\alpha^{ns+t})\right|<c^n$ for $n$ large enough and thus $(\lambda,\alpha)\in M\cap\overline{\Q}^2$. \hfill$\Box$
\medskip

\noindent
For $s\in\N^\times$ and real $\alpha>1$, let $\alpha^{1/s}$ denote the unique  
real number $\beta>1$ such that $\beta^s=\alpha$. We now have the following
corollary which is a result in the direction of the  fourth question indicated
in the introduction (see page 3).

\begin{cor}\label{corhardy+mahler+}
The following are true.\\
1) $(\lambda,\alpha)\in M\cap\overline{\Q}^2$
if and only if  there exists $s \in \N^\times$ and $ 0\le t<s$
such that $(\lambda\alpha^t,\alpha^s)\in H\cap\overline{\Q}^2$.\\
2) If $(\lambda,\alpha)\in H\cap\overline{\Q}^2,$ then 
for all $s\in\N^\times$ and  $0\le t<s,$ we have $(\lambda\alpha^{-t/s},\alpha^{1/s})\in M\cap\overline{\Q}^2$.\\
3) For any $s \in \N^\times,  (\lambda,\alpha) \in H $
if and only if $(\lambda\alpha^t,\alpha^s)\in H$ for every $0 \le t<s$.\\
4) For any $s\in\N^\times$, $ (\lambda,\alpha)\in M$  if and only if
there exists $0\le t<s$ such that $(\lambda\alpha^t,\alpha^s)\in M$.

\end{cor}
\begin{proof} Here is the proof of the above statements.

\smallskip
\noindent
 1) By Theorem~\ref{mahler+}, $(\lambda,\alpha)$ belongs to $M\cap\overline{\Q}^2$ if and only if there exists integers $0\le t<s$ such that $\alpha^s$ is a PV number and $\lambda\alpha^t\in\frac{1}{P'_{\alpha^{s}}(\alpha^{s})}\Z\left[\alpha^{s},\frac{1}{\alpha^{s}}\right]$, but these are exactly the conditions in Theorem~\ref{hardy+} for $(\lambda\alpha^t,\alpha^s)$ to belong to $H\cap\overline{\Q}^2$.

\smallskip
\noindent
2) It follows  from the previous assertion. For $(\lambda,\alpha)\in H\cap\overline{\Q}^2$ and
integers $s, t$ with  $0\le t<s$, let $(\mu,\beta)=(\lambda\alpha^{-t/s},\alpha^{1/s})$. Then $(\mu\beta^t,\beta^s)=(\lambda,\alpha)$. By the reverse implication in the above proposition, we get $(\mu,\beta)\in M\cap\overline{\Q}^2$.

\smallskip
\noindent
3) This follows from the definition of $H$ since for $s\in\N^\times$, each of
the following $s$  subsequences $$(\Vert \lambda\alpha^{ms+t}\Vert)_{m\in\N}, \phantom{ma} 0\le t<s$$ converges to $0$ if and only if the sequence $(\Vert \lambda\alpha^n\Vert)_{n\in\N}$ 
converges to $0$.

\smallskip
\noindent
4) For any $s\in\N^\times$ and  $(\lambda,\alpha)\in M$, there exists an integer $0\le t<s$ such that $\Vert\lambda\alpha^{ms+t}\Vert<c^m$ for infinitely many $m\in\N$ and thus $(\lambda\alpha^t,\alpha^s)\in M$. Reciprocally, if there exists an integer $0\le t<s$ such that $(\lambda\alpha^{t},\alpha^{s})\in M$, then $\Vert\lambda\alpha^{ms+t}\Vert<c^m$ for infinitely many $m\in\N$ and hence $(\lambda,\alpha)\in M$.
\end{proof}

\bigbreak

\smallskip
\noindent
{\bf Acknowledgments.} This work was initiated  during the visits
of the  second author to University of Paris VI and the first at IMSc. Both authors are grateful for the support under the  MODULI program as well as the Indo-French Program for Mathematics (belonging to CNRS) which facilitated these visits.

\bigskip

\address
[Patrice Philippon]
{Institut de Math\'ematiques de Jussieu,
UMR 7586 du CNRS,
4, place Jussieu,
75252 PARIS Cedex 05,
France.}\\
Email: {patrice.philippon@imj-prg.fr}
     
\address[Purusottam Rath]
     {Chennai Mathematical Institute,
     Plot No H1, SIPCOT IT Park,
     Padur PO, Siruseri 603103,
     Tamil Nadu, 
      India.}   \\
Email: {rath@cmi.ac.in}

\begin{thebibliography}{100}

\bibitem{aki} S. Akiyama and Y. Tanigawa,
{\it Salem numbers and uniform distribution modulo 1,}
Publ. Math. Debrecen 64 (2004), no. 3-4, 329-341. 

\bibitem{bugeaud} Y. Bugeaud and A.  Dubickas,
{\it On a problem of Mahler and Szekeres on approximation by roots of integers.}
Michigan Math. J. 56 (2008), no. 3, 703-715. 

\bibitem{cz}  P. Corvaja and U. Zannier, {\it On the rational approximations to the powers of an algebraic number: solution of two problems of Mahler and Mend\`es France}, Acta Math. {\bf 193} (2004), no. 2, 175-191.

\bibitem{de} B. de Smit, {\it Algebraic numbers with integral power traces}, J. Number Theory {\bf 45} (1993), no. 1, 112-116. 


\bibitem{doche} C. Doche, M. Mend\`es France and J.J. Ruch,
{\it Equidistribution modulo 1 and Salem numbers,}
Funct. Approx. Comment. Math. 39 (2008), part 2,

\bibitem{dubickas} A. Dubickas, {\it Integer parts of powers of Pisot and Salem numbers}, Arch. Math. {\bf 79} (2002), 252-257. 

\bibitem{hardy}  G. H. Hardy, {\it A problem of Diophantine approximation}, J. Indian Math. Soc. {\bf 11}, (1919), 205-243.

\bibitem{lang} S. Lang, {\it Algebraic Number Theory}, Graduate Texts in Math. {\bf 110}, Spinger, New York (1986).
 
\bibitem{langalgebra} S. Lang, {\it Algebra,} Revised Third Edition, Graduate Texts in Math. {\bf 211}, Spinger, New York (2005).
 
 \bibitem{mahler} K. Mahler, {\it On the fractional parts of the powers of a rational number, II }, Mathematika {\bf 4} (1957), 122-124.
 
 
\bibitem{see}  K. Mahler and G. Szekeres, {\it On the approximation of real numbers by roots of integers,} Acta Arith. 12 (1967), 315-320. 
 
 \bibitem{salem} R. Salem, {\it Algebraic numbers and Fourier analysis}, D. C. Heath and Co., Boston, Mass. (1963).

\bibitem{schmidt} W. M. Schmidt, {\it Diophantine Approximations and Diophantine Equations}, Lecture Notes in Math. {\bf 1467}, Springer (1991).

\bibitem{cetraro} W. M. Schmidt, {\it Linear Recurrence Sequences} in Diophantine Approximation, Lectures given at the C.I.M.E. Summer School held in Cetraro, Italy, June 28 - July 6, 2000, F. Amoroso \& U. Zannier eds, Lecture Notes in Math. {\bf 1819}, Springer (2003), 171-247.

\bibitem{zaimi} T. Zaimi, {\it On integer and fractional parts of powers of Salem numbers}, Arch. Math. {\bf 87} (2006), 124-128. 

\end{thebibliography}
\end{document}